\documentclass[12pt,oneside]{amsart}
          \usepackage{cite}
       \usepackage{amssymb}
 \usepackage{amsmath}
 \usepackage{amsthm}
       \usepackage[autostyle]{csquotes}
    \usepackage{hyperref}

      \theoremstyle{plain}
      \newtheorem{theorem}{Theorem}[section]
      \newtheorem{lemma}[theorem]{Lemma}
      \newtheorem{corollary}[theorem]{Corollary}
      \newtheorem{proposition}[theorem]{Proposition}
    
     \theoremstyle{definition}
      \newtheorem{definition}[theorem]{Definition}
      
      \newtheorem{remark}{Remark}

       \newcommand{\teq}{{\approx_\text t}}
      \newcommand{\nil}{\emptyset}
      \newcommand{\g}{{\mathfrak g}}
      \newcommand{\h}{{\mathfrak{h}}}
      
      \renewcommand{\a}{\mathfrak a}

       \newcommand{\E}{{\mathcal{E}}}
       \newcommand{\F}{{\mathcal{F}}}
      \renewcommand{\P}{{\mathcal{P}}}
     \newcommand{\Con}{\text{Con}}
      \newcommand{\n}[1]{\text l(#1)}
      \newcommand{\con}[2]{\Con(\mathfrak{#1},\mathcal{#2})}
       
      \newcommand{\tcon}[2]{\Con_{\text{t}}(\mathfrak{#1},\mathcal{#2})}

       \newcommand{\atom}[1]{\textbf{atom}(\mathcal{#1})} 
     
      \newcommand{\A}{\mathcal{A}}
       \newcommand{\B}{\mathcal{B}}

      \makeatletter
      \def\@setcopyright{}
      \def\serieslogo@{}
      \makeatother
   
\begin{document}

   \author{Ali Rejali}
   \address{University of Isfahan}  
   \curraddr{Department of Mathematics, Faculty of  Sciences , University of Isfahan, Isfahan, 81746-73441, Iran}
   \email{rejali@sci.ui.ac.ir}


   \author{Meisam Soleimani Malekan}
   \address{Department of Mathematics, Ph. D student, University of Isfahan, Isfahan, 81746-73441, Iran}
   \email{m.soleimani@sci.ui.ac.ir}

   \title{Two-sided configuration equivalence and isomorphism}

     \begin{abstract}
     The concept of configuration was first introduced to give a characterization for the amenability of groups. Then the concept of two-sided configuration was suggested to provide normality to study the group structures more efficiently. It has been interesting that for which groups, two-sided configuration equivalence would imply isomorphism. We introduce a class of groups, containing polycyclic and FC groups, which for them, the notions of two-sided configuration equivalence and isomorphism coincide. 
      \end{abstract}

   \subjclass[2010]{20F05, 20F16, 20F24, 43A07}

   \keywords{Configuration, Two-sided Configuration, Finitely presented Groups,  Hopfian Groups, FC--Groups, Polycyclic Groups, Nilpotent Groups}

   \thanks{The authors would like to express  their gratitude toward Banach Algebra Center of Excellence for Mathematics, University of Isfahan.}
   \thanks{}

   \dedicatory{}

   \date{\today}


   \maketitle



\section{Introduction and definitions}
In the present paper, all groups are assumed to be finitely generated. Let $G$ be a group, we denote the identity of the group $G$ by $e_G$. We refer the readers to \cite{combin} for terminology and statements used in finitely generated groups. \par   
  The notion of a configuration for a locally compact group $G$, was introduced in \cite{rw}. It was shown in that paper that amenability of $G$ may be characterized by configuration.  \par
Let {\small $\mathfrak{g} =(g_1,\dotsc, g_n)$} be an ordered generating set and {\small $\mathcal{E} =\{E_1,\dotsc,E_m\}$} be a finite partition of a group G. A \textit{configuration} $C$ corresponding to $(\g,\E)$, is an $(n+1)$-tuple $C=(c_0,\dotsc, c_n)$ whose components are in $\{1,\dotsc,m\}$ and there are $x_0, x_1,\dotsc, x_n$ in $G$ such that $x_k\in E_{c_k}$, $k=0,1,\dotsc, n$, and $x_k=g_kx_0$, $k=1,\dotsc,n$. In this case, we say that $(x_0,x_1,\dotsc,x_n)$ has the configuration $C$.\par
  For $\g$ and $\E$ as above, we call $(\g,\E)$ a \textit{configuration pair}. The set of all configurations corresponding to the configuration pair $(\g,\E)$ will be denoted by $\con gE$. The set of all configuration sets of $G$, is denoted by $\Con(G)$.\par 
 A group $G$ is \textit{configuration contained} in a group $H$, written $G\precsim H$, if $\Con(G)\subseteq\Con(H)$, and two groups $G$ and $H$ are \textit{configuration equivalent}, written $G\approx H$, if $\Con(G) =\Con(H)$. \par
 In \cite{arw}, the authors showed that finiteness and periodicity are the properties which can be characterized by configuration. In \cite{arw} and \cite{ary}, the authors proved that for two equivalent groups, the isomorphism classes of their finite or Abelian quotients were the same. Also, it was shown in \cite{arw} that two configuration equivalent groups should satisfy in the same semi-group laws, and in \cite{rs} we generalized this result by proving that same group laws should be established in configuration equivalent groups. Hence, in particular, being Abelian and the group property of being nilpotent of class $c$ are another properties which can be characterized by configuration (see \cite{arw} and \cite{ary}). In \cite{ary}, it was shown that if $G\approx H$, and $G$ is a torsion free nilpotent group of Hirsch length $h$, then so is $H$. It is interesting to know the answer to the question whether being FC-group is conserved by equivalence of configuration; in \cite{ary}, it was answered under the assumption of being nilpotent, but in \cite{rs} it was affirmatively answered without any extra hypothesis. In addition, it was shown in that paper that the solubility of a group $G$ can be recovered from $\Con(G)$. 
\par We also interested to investigate in which subclasses of the class \textbf{G} of all groups, does configuration equivalence coincide with isomorphism? In \cite{arw}, this question was answered positively for the class of finite, free and Abelian groups. In \cite{ary}, it was shown that those groups with the form of $\mathbb{Z}^n\times F$, where $\mathbb Z$ is the group of integers, $n\in\mathbb N$ and $F$ is an arbitrary finite group, are determined up to isomorphism by their configurations. In \cite{ary}, it was proved that if $G\approx D_\infty$, where $D_\infty$ is the infinite dihedral group, then $G\cong D_\infty$.
 \par In \cite{rs}, we pointed out that it was the existence of \enquote{golden} configuration pairs which implies isomorphism. Indeed, we showed that in the class of finitely presented Hopfian groups with golden configuration pair, configuration equivalence coincided with isomorphism. 
 \par  Consider a configuration pair $(\g,\E)$ for a group $G$, a \textit{two-sided configuration}
corresponding to this pair is an $(2n + 1)$-tuple $C= (c_0,c_1,\dotsc,c_{2n})$ satisfying $c_i\in\{1,\dotsc,m\}$, $i=0,1,\dotsc,2n$, and there exists $x\in E_{c_0}$ such that $g_ix\in E_{c_i}$ and $xg_i\in E_{c_{i+n}}$ for each $i\in\{1,\dotsc,n\}$. \par
We denote the set of all two-sided configurations corresponding to the configuration pair $(\g,\E)$ by $\tcon gE$, and the set of all configuration sets of $G$ by $\Con_{\text t}(G)$. \par
The concept of \textit{being two-sided configuration contained} and \textit{two-sided configuration equivalence} are defined in a similar manner, and denoted by symbols \enquote{$\precsim_{\text t}$} and \enquote{$\teq$}, respectively. 
\par Let $G$ be a group with a generating set $\g=(g_1,\dotsc,g_n)$ and a partition $\E=\{E_1,\dotsc, E_m\}$. Then $\g^{-1}:=(g_1^{-1},\dotsc, g_n^{-1})$ is another generating set of $G$. Also, if $g\in G$ then 
{\small $$\E^{-1}:=\{E_1^{-1},\dotsc, E_m^{-1}\}\quad\text{and}\quad \E g:=\{E_1g,\dotsc, E_mg\}$$}
are partitions of $G$. A simple calculation shows that $\con gE=\Con(\g,\E g)$ for every $g\in G$. Let $y\in G$ and $E_i\in \E$. Working with $\con gE$, we can assume that $y\in E_i$. This assumption is no longer true for two-sided configuration sets. Let $(\h,\F)$ be a configuration pair for a group $H$ with $\tcon gE=\tcon hF$. Then, equations $\con gE=\con hF$ and $\Con(\g^{-1},\E^{-1})=\Con(\h^{-1},\F^{-1})$ are hold and hence, if $G\teq H$ then $G\approx H$. We do not know whether the converse of this implication holds true or not? \par
We show that normality in a group is conserved under two-sided configuration equivalence. The conditions in the definition of golden configuration pair, \cite[Definition 4.1]{rs}, are difficult to be investigated. We can reduce these conditions in the context of two-sided configuration to get a new definition of golden configuration pair, and show that there are many groups satisfied in our definition. 
 We show that if $G$ and $H$ have the same two-sided configuration sets and $N$ is a normal subgroup of $G$ such that $G/N$ is finitely presented and has a golden configuration pair, then there is a normal subgroup $\mathfrak N$ of $H$ such that $G/N\cong H/ \mathfrak N$; as a consequence of this fact, for the class of groups which are finitely presented and Hopfian with golden configuration pair, two notions of two-sided configuration equivalence and isomorphism coincided.  This class of groups, contains many well-known class of groups, such as finite, free, abelian, nilpotent, FC, and polycyclic groups. We will define a new class of groups, called \textit{polynomial type groups}, which have golden configuration pairs. Also, all groups which have a finite subnormal series such that all its factors are polynomial type admit a golden configuration pair.

\section{Preliminaries}
We devote this section to provide the preliminaries and notations needed in the following sections;\par 
\par Following \cite{rs}, let $G$ and $H$ be two groups with generating sets $\g$ and $\h$, respectively. Suppose that, for partitions 
 $\E=\{E_1,\dotsc,E_m\}$ and $\F=\{F_1,\dotsc,F_m\}$ of $G$ and $H$ respectively, the equality 
 $\tcon gE=\tcon hF$, or $\con gE=\con hF$, established. Then we may say that $E_i$ is \textit{corresponding} to $F_i$, and write $E_i \leftrightsquigarrow F_i$, $i=1,\dotsc,m$.
 \par
Let $G$ be a group with $\mathfrak{g}=(g_1,\dotsc,g_n)$ as its ordered generating set. Let $p$ be a positive integer, let  $J$ and $\rho$ be $p$-tuple with components in $\{1,2,\dotsc,n\}$ and $\{\pm1\}$, respectively. We denote the product 
$\prod_{i=1}^pg_{J(i)}^{\rho(i)}$ by $W(J,\rho;\mathfrak{g})$. We call the pair 
$(J,\rho)$ a \textit{representative pair} on $\g$ and $W(J,\rho;\mathfrak{g})$ a \textit{word} corresponding to $(J,\rho)$ in $\g$. \par
The reader confirms that if $\g$ in the above notation is an ordered subset of $G$ instead of being a generating set, then $W(J,\rho;\mathfrak{g})$ can be defined as well. \par 
For an arbitrary multiple $J$, we denoted its components number by $\n J$.When we speak of a representative pair $(J,\rho)$ we assume the same number of components for $J$ and $\rho$. If $J=(J(1),\dotsc,J(p))$, and $\rho=(\rho(1),\dotsc,\rho(p))$, where $p$ is a positive integer, we set
 \begin{align*}
 J^{-1}&:=(J(p),\dotsc,J(1))\quad\text{and}\quad
 \rho^{-1}:=(-\rho(p),\dotsc,-\rho(1))
\end{align*}
\par
For positive integers $p_i$, if $J_i$ is a $p_i$-tuple, $i=1,2$, $J_1\oplus J_2$ is a $(p_1+p_2)$-tuple that has $J_1$ as its first $p_1$ components, and $J_2$ as its second $p_2$ components. it can be easily investigated that
\begin{small}
\begin{enumerate}
\item $W(J_1,\rho_1;\mathfrak{g})W(J_2,\rho_2;\mathfrak{g})=
W(J_1\oplus J_2,\rho_1\oplus\rho_2;\mathfrak{g})$
\item $W(J,\rho;\mathfrak{g})^{-1}=W(J^{-1},\rho^{-1};\mathfrak{g})$
\end{enumerate}
\end{small}
For a $\sigma$-algebra $\A$, we define $\tcon gA$ to be $\Con_\text t(\g,\atom A)$, where $\atom A$ is the collection of atomic sets in $\A$. 
\par For a finite collection $\mathcal{C}$ of subsets of $G$, the $\sigma$-algebra generated by elements of $\mathcal{C}$ is denoted by $\sigma(\mathcal C)$ and is finite.\par 
 Let $\E:=\{E_1,\dotsc,E_m\}$ and $\F:=\{F_1,\dotsc,F_m\}$ be partitions of $G$ and $H$ respectively, such that $E_i\leftrightsquigarrow F_i$, $i=1,\dotsc,m$. For $A\in\sigma(\E)$ and $B\in\sigma(\F)$, we say $A$ is corresponding to $B$, written $A\leftrightsquigarrow B$, when 
$$\{k:\,E_k\cap A\neq\nil\}=\{k:\, F_k\cap B\neq\nil\}$$
If $A\leftrightsquigarrow B$, and $A=E_{i_1}\cup\dotsb\cup E_{i_j}$, then $B=F_{i_1}\cup\dotsb\cup F_{i_j}$. By an argument as used in Lemma 2.2 of \cite{rs}, one can show that 
\begin{lemma}
\label{important}
Let $G$ and $H$ be two groups with finite $\sigma$-algebras $\A$ and $\B$ and generating sets $\g=(g_1,\dotsc,g_n)$ and $\h=(h_1,\dotsc,h_n)$, respectively, such that $\tcon gA=\tcon hB$. Consider $A_1,A_2\in\A$ and $B_1, B_2\in\B$ with $A_i\leftrightsquigarrow B_i$, $i=1,2$. We have:
\begin{itemize}
\item[(1)] If $g_rA_1\subseteq A_2$, then $h_rB_1\subseteq B_2$, 
\item[(2)] If $A_1g_r\subseteq A_2$, then $B_1h_r\subseteq B_2$, 
\item[(3)] If $g_rA_1=A_2$, then $h_rB_1=B_2$, 
\item[(4)] If $A_1g_r=A_2$, then $B_1h_r=B_2$, 
\end{itemize}
for $r\in\{1,\dotsc,n\}$.
\end{lemma}
Let $G$ and $H$ be two groups. Consider partitions $\E=\{E_1,\dotsc,E_r\}$ and $\F=\{F_1,\dotsc,F_r\}$ of $G$ and $H$, respectively, and their refinements
$$
\E'=\{E'_1,\dotsc,E'_s\}\,\text{and}\,\F'=\{F'_1,\dotsc,F'_s\}
$$
We say that $(\E',\E)$ and $(\F',\F)$ are \textbf{similar} and write 
$(\E',\E)\sim(\F',\F)$, if 
$$ \{l : E_k\cap E'_l\neq\emptyset\}=\{l : F_k\cap F'_l\neq\emptyset\}\quad(k=1,\dotsc,r) $$
In other words, if $E_k=\bigcup_{j=1}^tE'_{i_j}$, then we have $F_k=\bigcup_{j=1}^tF'_{i_j}$. \par 
For $\sigma$-subalgebras $\A'$ and $\B'$ of $\sigma$-algebras $\A$ and $\B$, respectively, we say $(\A,\A')$ and $(\B,\B')$ are \textbf{similar}, written $(\A,\A')\sim(\B,\B')$, if 
$$(\atom A,\atom{\A'})\sim(\atom{B},\atom\B')$$
A similar argument as in used in the proof of \cite[Lemma 3.2.]{rs}, ensures that:
\begin{lemma}
\label{Similarity}
let $G$ and $H$ be two groups. Assume that $(\g,\E)$ and $(\h,\F)$ are two configuration pairs for $G$ and $H$, respectively, and let $\E'$ and $\F'$ be their similar refinements such that $\tcon g{E'}=\tcon h{F'}$. Then $\tcon gE=\tcon hF$.
\end{lemma}

\section{Two-sided configuration of finitely presented Hopfian groups}
Working with two-sided configuration, we will see that we are provided with a normality in the following sense: 
\begin{definition}
\label{normality}
Let $G$ be a group and $E\subseteq G$. We say that $E$ is a \textit{normal subset} of $G$, if 
\begin{equation}
\label{31}
Eg=gE\quad(g\in G)
\end{equation}
\end{definition}
It is obvious that if $\g=(g_1,\dotsc,g_n)$ is a generating set of $G$, then we can reduce equations in (\ref{31}) to 
$$Eg_i=g_iE\quad(i=1,\dotsc,n)$$
Let $G$ be a group and, $N$ be a normal subgroup of $G$. We denote by $q:=q_N$ the quotient map. For a partition $\E$ of $G/N$ and a generating set $\g=(g_1,\dotsc,g_n)$ of $G$, it is evident that $q^{-1}(\E):=\{q^{-1}(E):\,E\in\E\}$ is a partition of $G$ and $q(\g)=(q(g_1),\dotsc,q(g_n))$ is a generating set of $G/N$. \par 
We know by \cite[Theorem 1.1]{combin} that two finitely generated groups are isomorphic, if they satisfy in the same relations. So, if for a group $G$, two-sided configuration equivalence leads to isomorphism, then at least one of the presentations of each non-identity element should be recovered by the configuration set of the group. This recovery is precisely what we mean by a \enquote{golden configuration pair}:
\begin{definition}
\label{recognizable}
Let $G$ be a group with a generating set $\g$. Consider a normal subgroup $N$ of $G$ and a partition $\E$ of $N$. We say that $(\g,q^{-1}(\E))$ is \textit{golden} w.r.t. $N$, if $\E$ contains $\{e_{G/N}\}$ and there exist representative pairs $(J_g,\rho_g)$, $g\in G\setminus N$, such that 
\begin{itemize}
\item[(1)] $q(g)=W(J_g,\rho_g;q(\g))$,
\item[(2)] If $\Con_\text t(\g,q^{-1}(\E))=\tcon  hF$, for a configuration pair $(\h,\F)$ of a groups $H$, then 
$$W(J_g,\rho_g;\h)M\cap M=\nil$$
where $M\in\F$ is corresponding to $N$.
\end{itemize}
 \end{definition}
 We say that a configuration pair $(\g,\E)$ of $G$ is golden when $(\g,\E)$ is a golden configuration pair w.r.t. $N=\{e_G\}$. \par 
 As a consequence of Lemma \ref{Similarity}, one can easily show that (see \cite[Lemma 3.3.]{rs}):
\begin{lemma}
\label{refinement}
Let $G$ be a group with a normal subgroup $N$. Assume that $(\g,q^{-1}(\E))$ is a golden configuration pair w.r.t. $N$. Then for each refinement $\E'$ of $\E$, the configuration pair $(\g,q^{-1}(\E'))$ is golden w.r.t. $N$. 
\end{lemma}
Using this lemma, we have:
\begin{lemma}
\label{all representative pair instead of one}
Let $N$ be a normal subgroup of a group $G$ such that $G/N$ is finitely presented. Suppose that $G$ has a golden configuration pair, $(\g,q^{-1}(\E))$, w.r.t. $N$. Then there is a refinement $\E'$ of $\E$ such that if $\Con_\text t(\g,q^{-1}(\E'))=\tcon  hF$, for a configuration pair $(\h,\F)$ of a groups $H$, and $M\in\F$ is corresponding to $N$, then $M$ will be a normal subset and 
$$W(J,\rho;\h)M\cap M=\nil$$
when $W(J,\rho;\g)\not\in N$.
\end{lemma}
\begin{proof}
By Lemma \ref{refinement}, it suffices to show that there exists a refinement $\E'$ of $\E$, such that the map $W(J,\rho;\g)N\mapsto W(J,\rho;\h)M$ will be a well-defined function on the cosets of $N$. \par 
Suppose that $\g=(g_1,\dotsc,g_n)$. Let $\{W(J_i,\sigma_i;q(\g)):\, i=1,\dotsc,m\}$ be a set of defining relators of $G/N$, and 
$$\mathfrak{F}:=\{(J_i,\sigma_i):\, i=1,\dotsc,m\}\cup\{((j),(1)):\,j=1,\dotsc,n\}$$
Assume that $k:=\max\{\n J:\,(J,\rho)\in\mathfrak{F}\}$, and set
\begin{align*}
S:&=\{(J,\rho):\,\n J\leq 3k\}\\
S_0:&=\{(J,\rho):\,\n J\leq k\}
\end{align*}
Now, consider a refinement $\E'$ which contains singleton sets $\{W(J,\rho;q(\g))\}$, $(J,\rho)\in S$. Let $(\h,\F)$ be a configuration pair  for a group $H$, such that $\Con_\text t(\g,q^{-1}(\E'))=\tcon hF$ and consider $M\in\F$ to be a set that $N\leftrightsquigarrow M$. We see that $Ng_i=N=g_iN$, $i=1,\dotsc,n$. By (3) and (4) in Lemma \ref{important}, $M$ is a normal subset of $H$. Also, by induction on $\n J$, one can show that 
$$W(J,\rho;\h)M=M(J,\rho),\quad (J,\rho)\in S_0$$
where $M(J,\rho)$ is an element of $\F$ corresponding to $W(J,\rho;\g)N$. So, in particular, we have 
$$W(J_i,\rho_i;\h)M=M,\quad i=1,\dotsc,m$$
Let $(J,\rho)$ and $(I,\delta)$, be two representative pair, and take $W_i=W(J_i,\rho_i;\h)$, $i=1,\dotsc,m$. As $M$ is normal, and $W_iM=M$, it can be seen that 
$$W(J,\rho;\h)W_i^{\pm}W(I,\delta;\h)M=W(J\oplus I,\rho\oplus\delta;\h)M$$
This shows that the above function is well-defined. 
\end{proof}
By this lemma, we state and prove the key lemma of the paper:
\begin{lemma}
\label{4.4}
Let $G$ and $H$ be two groups such that $G\teq H$. Suppose that for a normal subgroup $N$ of $G$ the quotient $G/N$ is finitely presented and $G$ has a golden configuration pair w.r.t. $N$. Then, there is a normal subgroup $\mathfrak{N}$ of $H$ such that $G/N\cong H/\mathfrak{N}$.
\end{lemma}
\begin{proof}
Assume that $(\g,q^{-1}(\E))$ is a golden configuration pair w.r.t. $N$. Suppose that $\Con_\text t(\g,q^{-1}(\E))=\tcon  hF$ and $N\leftrightsquigarrow M\in\F$. By the preceding lemma, we can assume that $M$ is a normal subset and $W(J,\rho;\h)M\cap M=\nil$ when $W(J,\rho;\g)\not\in N$. Define 
$$\mathfrak{N}=H_M:=\{h\in H:\,hM=M\}$$
By normality of $M$, we understand that $\mathfrak N$ is a normal subgroup of $H$. Moreover, one can check that for a representative pair $(J,\rho)$, $W(J,\rho;\h)M\cap M\neq\nil$ if and only if $W(J,\rho;\g)N=N$, and the last one is equivalent to $W(J,\rho;\h)\in \mathfrak N$. We claim that $M$ is a coset of $\mathfrak{N}$; first, suppose that $M$ intersects two different cosets of $\mathfrak{N}$, so there are $x$ and $y$ in $M$ with $yx^{-1},y\not\in\mathfrak{N}$. Let $(J,\rho)$ be a representative pair such that $yx^{-1}=W(J,\rho;\h)$. Then $y\in W(J,\rho;\h)M\cap M$, therefore $W(J,\rho;\h)\in\mathfrak N$, this is a contradiction. So, there is an $x\in H$ such that $M\subseteq\mathfrak Nx$. If $h\in\mathfrak N$, then $h$ fixes $Mz^{-1}$, $z\in M$, under the left multiplication. But the last set contains $e_H$, hence $h\in Mz^{-1}$, and therefore $\mathfrak Nz\subseteq M\subseteq\mathfrak Nx$. This implies that $M=\mathfrak{N}x$.
\par The map
$$G\rightarrow H/\mathfrak{N},\quad W(J,\rho;\g)\mapsto W(J,\rho;\h)\mathfrak{N}$$
along with the explanation at the beginning of the proof gives an epimorphism which has $N$ as its kernel; so, $G/N$ and $H/\mathfrak{N}$ are isomorphic.
\end{proof}
\begin{remark}
\textbf{(a)} We proved that $M=\mathfrak{N}x$ for an $x\in H$. By normality of $M$, we conclude that $\mathfrak{N}x^h=\mathfrak{N}x$, for every $h\in H$. Therefore $q_\mathfrak N(x)\in Z(H/\mathfrak N)$, where $Z$ stands for the center of the group $H/\mathfrak N$.
\par \textbf{(b)} If $M$ contains the identity of $H$, then by the above Lemma, $\mathfrak{N}$ is nothing but $M$. This means that $M$ should be a subgroup of $H$. But there is no reason to say that $M$ should contain $e_H$, so it is not true to consider such an assumption, as it was in proofs of statements in the last section of \cite{ry}.\par 
\textbf{(c)} Lemma \ref{4.4}, is indeed a generalization of \cite[Lemma 6.3]{arw} and \cite[Proposition 2.2]{ary}, absolutely in the context of two-sided configuration. The main consequence of this lemma is:
\par 
Let $G$ be a finitely presented group with a golden configuration pair. Assume that for a group $H$ we have $G\approx_\text t H$. Then $G$ is isomorphic to a quotient of $H$. In addition, if $M$ is a singleton subset of $H$, then $G\cong H$.
 \end{remark}
Now, we state and prove the main theorem of this section (compare with \cite[Proposition 3.5.]{rs}):
\begin{theorem}
\label{FP}
Let $G$ be a finitely presented Hopfian group with a golden configuration pair, and $H$ be a group which $G\approx_{\text t}H$. Then $G\cong H$.
\end{theorem}
\begin{proof}
Let $(\g,\E)$ be a golden configuration pair. Assume that $\tcon gE=\tcon hF$, for a configuration pair $(\h,\F)$ of $H$. Let $M\in\F$ be such that $\{e_G\}\leftrightsquigarrow M$. Without loss of generality, we can assume that $M$ is a normal subset and $W(J,\rho;\h)M\cap M=\nil$ when $W(J,\rho;\g)\not\in N$.\\
 By Remark 1.\textbf{(c)}, if $M$ is singleton, then we've done. Seeking a contradiction, let us suppose that $M=M_1\cup M_2$, where $M_1$ and $M_2$ are nonempty disjoint subsets. Consider the following partition of $H$, 
$$\F'=\{M_1,M_2\}\cup(\F\setminus\{M\})$$
There is a configuration pair $(\g',\mathcal K')$ of $G$ such that $\tcon h{F'}=\tcon {g'}{K'}$. Make $\mathcal K'$ coarser to get a partition $\mathcal K$ such that $(\mathcal K',\mathcal K)\sim(\F',\F)$. By Lemma \ref{Similarity}, 
$$\tcon gE=\tcon hF=\tcon {g'}{K}$$
 If $\{e_G\}\leftrightsquigarrow K$, for $K\in\mathcal K$, then $K=K_1\cup K_2$, where $K_i\leftrightsquigarrow M_i$, $i=1,2$. 
Since $G$ is Hopfian, by Lemma \ref{4.4} and Remark 1.\textbf{(c)}, we get 
$K_G=\{g\in G: gK=K\}=\{e_G\}$, and this is equivalent to $|K|=1$, which is impossible. 
\end{proof}
In the next section, we will see that there are many groups with golden configuration pair.

\section{Groups with golden configuration pairs}
 Let $\varsigma$ be a finite-range function on a group $G/N$. we will denote the range of $\varsigma$ by $\{\varsigma_1,\dotsc,\varsigma_m\}$. Set $E(\varsigma_i):=q^{-1}(\varsigma^{-1}(\varsigma_i))$, $i=1,\dotsc,m$. The partition $\E:=\{E(\varsigma_i):\,i=1,\dotsc,m\}$ of $G$ is called the\textit{ $\varsigma$-partition} of $G$.
 \begin{definition} \label{golden systems}
 Let $N$ be a normal subgroup of a group $G$. We say that $G$ admits a \textit{golden system} w.r.t. $N$, if there exist a generating set $\g$ of $G$, a set of representative pairs, $\{(J_g,\rho_g):\,g\in G\setminus N\}$, and a finite-range function $\varsigma$ on $G/N$ with following properties: 
 \begin{enumerate}
 \item $q(g)=W(J_g,\rho_g;q(\g))$, $g\in G\setminus N$,
 \item  $E(\varsigma(e_{G/N}))=N$, 
 \item Let $\E$ be the $\varsigma$-partition of $G$. If we have $\tcon gE=\tcon hF$, for a configuration pair $(\h,\F)$ of a group $H$, and we denote by $F(\varsigma_i)$ an elements in $\F$ corresponding to $E(\varsigma_i)$, $i=1,\dotsc,m$, then 
 $$W(J_g,\sigma_g;\h)F(\varsigma(e_{G/N}))\subseteq F(\varsigma(q(g)))$$
 \end{enumerate}
 We call the triple $(\g,\varsigma,\{(J_g,\rho_g)\}_g)$ \textit{the golden system w.r.t. $N$}and $(J_g,\rho_g)$ \textit{the golden representative pair} of $q(g)$. \\
We say that a group $G$ admits a golden system, if it admits a golden system w.r.t. $\{e_G\}$. 
\end{definition}
It is evident that a group $G$ admits a golden system w.r.t. a normal subgroup $N$, if and only if $G/N$ admits a golden system. \par
The importance of this definition will appear in the following proposition:
\begin{proposition}
\label{golden system and golden configuration pair}
Let $N$ be a normal subgroup of a group $G$, and $(\g,\varsigma,\{(J_g,\rho_g)\}_g)$ be a golden system w.r.t. $N$. If $\E$ is the $\varsigma$-partition of $G$, then $(\g,\E)$ will be a golden configuration pair w.r.t. $N$.
\end{proposition}
\begin{proof}
$\E$ is in the form of $q^{-1}(\mathcal{P})$, for the partition 
$$\P=\{\{q(g):\varsigma(q(g))=\varsigma_i\}:\, i=1,\dotsc,r\}$$
of $G/N$. It is evident by this point that $(\g,\E)$ becomes a golden configuration pair.
\end{proof}
We list below classes of groups which admits a golden systems:
\subsection{Finite groups}
\label{finite group and golden system}
A group $G$ with a finite quotient, say $G/N$, admits a golden system w.r.t. $N$. By \cite[Lemma 6.2.]{arw}, we know that there is a generating set $\g=(g_1,\dotsc,g_n)$ of $G$ such that $G/N=\{q(g_1)=e_{G/N},q(g_2),\dotsc,q(g_n)\}$. For $g\in G$, if $q(g)=q(g_i)$, for the smallest index $i$, $i=1,\dotsc,n$, then we will take $J_{g}=i$, $\rho_{i}=1$, and $\varsigma(q(g))=i$. we claim that $(\g,\varsigma,\{(J_{g},\rho_{g})\}_g)$ is a golden system. We have $q(g)=q(g_i)=W(J_g,\rho_g;q(\g))$, $g\in G$, and, $E(\varsigma(g))=g_iN=gN$. Since 
 $$W(J_g,\rho_g;\g)E(\varsigma(e_{G/N}))=g_iN=E(\varsigma(q(g)))$$
 then by Lemma \ref{important}, (3) in the Definition \ref{golden systems} holds.\par 
In particular, all finite groups admit a golden system.

\subsection{Polycyclic groups}
\label{Polycyclic group and golden system}
Terminologies of polycyclic groups are as in \cite{sims}.\\ 
Let $A$ be a polycyclic group, and $A=A_1\unrhd\dots\unrhd A_{n+1}=\{e_A\}$ be a polycyclic series for $A$. For $1\leq i\leq n$, let $a_i$ be an element of $A_i$ whose image in $A_i/A_{i+1}$ generates
that group. The sequence $\mathfrak a:=(a_1,\dotsc, a_n)$ will be called a \textit{polycyclic generating set} of $A$. By the proof of {\small \cite[Chapter 9; Proposition 3.5]{sims}}, $A_i=\langle a_i,\dotsc, a_n\rangle$ and every element $a$ of $A_i$ can be expressed in the form $\prod_{j=i}^n a_j^{\alpha_j}$, where the exponents $\alpha_j$ are integers. Let
$\mathfrak I$ denote the set of subscripts $i$ such that $A_i/A_{i+1}$ is finite, and let $m_i=|A_i : A_{i+1}|$, the order of $a_i$ relative to $A_{i+1}$, if $i$ is in $\mathfrak I$. It shall be normally assumed that the generating set is not redundant in the sense that no $a_i$ is in $A_{i+1}$. Thus $m_i>1$ for each $i\in\mathfrak I$. We shall say that the
expression for $a$ is a \textit{collected word} if $0\leq\alpha_j\leq m_j$ for $j$ in $\mathfrak I$, and the representative pair corresponds to a collected word called a \textit{collected representative pair}. Each element of $A$ can be described by a unique collected word in the generators $a_1,\dotsc, a_n$. \par
 Assume that $(J_a,\rho_a)$ is the collected representative pair of $a\in A\setminus\{e_A\}$. For $a\in A\setminus\{e_A\}$, let $\prod_{j=1}^n a_j^{\alpha_j}$ be its collected word, set $\varsigma(a)=(\omega_1,\dotsc,\omega_n)$, where $\omega_j=\alpha_j$, if $j\in\mathfrak I$, and for $j$ not in $\mathfrak I$, $\omega_i$ is 0, 1 or $-1$, depending on whether $\alpha_j$ is zero, positive or negative, respectively. Also, put $\varsigma(e_G)=\textbf{o}$, the $n$-tuple whose all components are zero. It is obvious that there are finitely many elements in $\varsigma(A)$, say $\varsigma_0,\varsigma_1,\dotsc, \varsigma_m$, where $\varsigma_0=\varsigma(e_G)$.\par We claim that $(\mathfrak a,\varsigma,\{(J_a,\rho_a)\}_a)$ is a golden system for $A$. Suppose that $\tcon aE=\tcon bF$, for a configuration pair $(\mathfrak{b},\F)$ of a group $B$, and we denote by $F(\varsigma(a))$ an element in $\F$ corresponding to $E(\varsigma(a))$, $a\in A$.\\
Let $\prod_{j=1}^n a_j^{\alpha_j}$ be the collected word of $a\in A\setminus\{e_A\}$. For $x\in A$, with collected word $\prod_{j=i}^n a_j^{\alpha_j}$, we have:
\begin{enumerate}
\item[(i)] If $i\in\mathfrak I$, then $a_iE(\varsigma(x))=E(\varsigma(a_ix))$,
\item[(ii)] If $k<i$ with $k\in\mathfrak I$, then $a_kE(\varsigma(x))=E(\varsigma(a_kx))$,
\item[(iii)] If $k<i$ with $k\not\in\mathfrak I$, then $a_k\left(E(\varsigma(x))\cup E(\varsigma(a_kx))\right)= E(\varsigma(a_kx))$.
\item[(iv)] If $i\not\in\mathfrak I$, and $\alpha_i<0$, then 
$$ a_iE(\varsigma(x))=E(\varsigma(x))\cup E\left(\varsigma\left(\prod_{j=i+1}^na_j^{\alpha_j}\right)\right),$$
\item[(v)] If $i\not\in\mathfrak I$, and $\alpha_i>0$, then $$a_i\left(E(\varsigma(x))\cup E\left(\varsigma\left(\prod_{j=i+1}^na_j^{\alpha_j}\right)\right)\right)=E(\varsigma(x)).$$
\end{enumerate} 
By (i)-(v), induction on $\n{J_a}$ and using Lemma \ref{important}, we can easily prove that 
\begin{equation*}
W(J_a,\rho_a;\mathfrak{b})F(\varsigma(e_A))\subseteq F(\varsigma(a)).
\end{equation*}
It is not difficult to show that if  a group $G$ has a normal subgroup $N$ such that $G/N$ is polycyclic, then $G$ will admit a golden system w.r.t. $N$.

\subsection{Polynomial type groups}\label{4.3}
Let $P:=P(x_1,\dotsc,x_n)$ be a polynomial in $\mathbb Z[x_1,\dotsc, x_n]$ without a constant term and without negative coefficients. Let $\{G_1,\dotsc, G_n\}$ be a collection of groups each of them is polycyclic or finite. We denote by $P(G_1,\dotsc, G_n)$ a group obtained from $P$, by putting $G_i$ instead of $x_i$, $i=1,\dotsc,n$, free product instead of multiplication and direct product instead of addition. For example, set $P(x,y)=xy^3+2x^2y$, then for polycyclic or finite groups $G$ and $H$, $P(G,H)$ denotes the group \small{$(G*H*H*H)\oplus(G*G*H)\oplus(G*G*H)$}. We mean such a group by a \textit{polynomial type group}.  \par 
Consider groups $G$ and $H$ with golden systems $(\g,\varsigma_G,\{(J_g,\rho_g)\}_g)$ and $(\h,\varsigma_H,\{(J_h,\rho_h)\}_h)$, resp. The following two procedures will be used in obtaining a golden system for polynomial type groups:\par
\textbf{1. Free product.} For a non-identity element $z$ in $G*H$, let $z_1\dots z_s$ be its reduced word expression in $G$ and $H$. Then define $\varsigma_{G*H}(z)$ to be $\varsigma_G(z_1)$ or $\varsigma_H(z_1)$ according to $z_1\in G$ or $z_1\in H$, respectively. For the identity element, put $\varsigma_{G*H}(e_{G*H})=0$. \par 
\textbf{2. Direct product.} For $(g,h)\in G\oplus H$, set $\varsigma_{G\oplus H}(g,h):=(\varsigma_G(g),\varsigma_H(h))$. 
\par A polynomial type group admits a golden system. First Note that:
\par Free product of finitely many poycyclic or finite groups admits a golden system; let $A_1,\dotsc, A_r$ be polycyclic groups and $F_1,\dotsc,F_s$ be finite groups. Set {\footnotesize $G=A_1*\dotsb*A_r*F_1*\dotsb*F_s$}. Assume that $(\mathfrak a_i,\varsigma_{A_i},\{(J_a,\rho_a)\}_a)$ is a golden system of $A_i$, $i=1,\dotsc,r$, gained in \ref{Polycyclic group and golden system}, and $(\mathfrak b_{j},\varsigma_{F_j},\{(J_x,\rho_x)\}_x)$ is a golden system of $F_j$, $j=1,\dotsc,s$ as obtained in \ref{finite group and golden system}. Hence, 
$$\g=(g_1,\dotsc,g_n):=\mathfrak a_1\oplus\dotsb\oplus\mathfrak a_r\oplus\mathfrak b_1\oplus \dotsb\oplus \mathfrak{b}_s$$
 is a generating set of $G$. Iterating the first procedure above, for free product, we get a function $\varsigma$ on $G$. For a non-identity $g\in G$, let $z_1\dotsb z_k$ be its reduced word expression, and consider the representative pair $(J_g,\rho_g)$ such that 
\begin{align*}
W(J_g,\rho_g;\g)=\prod_{i=1}^kW_i.
\end{align*} 
where $W_i$, $i=1,\dotsc,k$, is the word corresponding to the golden representative pair of $z_i$. Then, like the argument used in \ref{Polycyclic group and golden system}, one can show that \small{$(\g,\varsigma_G,\{(J_g,\rho_g)\}_g)$} is a golden system of $G$.\\
Using the second procedure, and the preceding paragraph, it can be easily seen that a polynomial type group admits a golden system.
\subsection{Groups which have a subnormal series such that its factors are polynomial type}
\label{Subnormal series with polycyclic or finite factors} 
let $G$ be a group which has a subnormal series, 
\begin{equation}\label{subnormal series}
 N_k=\{e_G\}\unlhd N_{k-1}\unlhd\dotsb\unlhd N_0=G
\end{equation}
each of its factor is polynomial type. Suppose that {\small $q_i:N_i\rightarrow N_i/N_{i+1}$} is the quotient map and $\a_i$ is an ordered subset of $N_i$, such that 
$$\left(q_i(\a_i),\varsigma^{(i)},\left\{\left(J_{q_i(g)},\rho_{q_i(g)}\right)\right\}_g\right)$$
is a golden system of $N_i/N_{i+1}$, where the index $g$ ranges over $N_i\setminus N_{i+1}$, $i=0,\dotsc,k-1$. One can see that $\g_i:=\a_i\oplus\dotsb\oplus\a_{k-1}$  is a generating set of $N_i$, $i=0,\dotsc,k-1$; for $g\in N_i$, put $g_1=g$, take an index $i_1$ such that $g\in N_{i_1}\setminus N_{i_1+1}$, so, there is $g_2\in N_{i_1+1}$ with 
$$g_1=W(J_{q_{i_1}(g_1)},\rho_{q_{i_1}(g_1)};\a_{i_1})g_2.$$
Continuing this process, we can find $g_3$, $g_4$, ..., $g_s$ and indices $i_1<i_2<\dotsb<i_s$, such that 
$$g_1=\prod_{j=1}^sW(J_{q_{i_j}(g_j)},\rho_{q_{i_j}(g_j)};\a_{i_j}).$$
Therefore, $\g=\g_0:=(g_1,\dotsc,g_n)$ is a generating set of $G$. Now, for a non-identity $g\in G$, let 
$$g=\prod_{j=1}^rW\left(J_{q_{i_j}(g_j)},\rho_{q_{i_j}(g_j)};\a_{i_j}\right) $$
be an expression, obtained by the above procedure, so $i_1<i_2<\dotsb<i_r$, and $g_j\in N_{i_j}\setminus N_{i_j+1}$. Define 
$$\varsigma(g):=\varsigma^{(i_1)}(q_{i_1}(g_1))$$
and consider representative pair $(J_g,\rho_g)$ so that 
\begin{align*}
W(J_g,\rho_g;\g):=\prod_{i=1}^r W\left(J_{q_{i_j}(g_j)},\rho_{q_{i_j}(g_j)};\a_{i_j}\right)
\end{align*}
By part \ref{4.3}, we see that $(\g,\varsigma,\{(J_g,\rho_g)\})$ is a golden system for $G$. \par 

\subsubsection{Polycyclic-by-finite groups}
Recall that a polycyclic-by-finite group, is a group which has a polycyclic normal subgroup of finite index, hence, by the above explanation, every polycyclic-by-finite group admits a golden system.
 
\subsubsection{$FC$-groups}
If a finitely generated group $G$ has a finite commutator subgroup, then $G$ admits a golden system by \ref{Subnormal series with polycyclic or finite factors}. Since $FC$-groups have finite commutator subgroups, every $FC$-group admits a golden system.
\par
We now state the main results of this paper, obtained from Lemma \ref{4.4}, Theorem \ref{FP}, Proposition \ref{golden system and golden configuration pair} and the fact that all poycyclic or finitely generated $FC$ groups are finitely presented:
\begin{theorem}
Let $G$ be a finitely presented Hopfian group with a golden system. Then a group $H$ which $G\teq H$ is isomorphic to $G$.  
\end{theorem}

\begin{corollary}(a) Let $G$ be either a polycyclic or an $FC$ group, and $G\teq H$ for a group $H$. Then $G\cong H$.
\par (b) Let $N$ be a normal subgroup of a group $G$ such that the quotient $G/N$ is either a polycyclic or an $FC$ group. If $G\teq H$ for a group $H$, then there exist a normal subgroup $\mathfrak{N}$ of $H$ such that $G/N\cong H/\mathfrak{N}$.
\end{corollary}
\begin{remark}
In \cite[Theorem 4.3]{rs1}, we proved that there exist non-isomorphic soluble groups $G$ and $H$ with same two-sided configuration sets. 
\end{remark}

 \section*{Acknowledgement}
The authors would like to express  their gratitude toward Banach Algebra Center of Excellence for Mathematics, University of Isfahan.


\begin{thebibliography}{0}

\bibitem{arw}
        A. Abdollahi, A. Rejali and G. A. Willis, \emph{Group Properties Characterized by                      Configurations}. Illinois J. Math. 2004. 48(3): 861-873.

\bibitem{ary} 
       A. Abdollahi, A. Rejali and A. Yousofzadeh, \emph{Configuration of Nilpotent Groups and
       Isomorphism}. J. Algebra Appl. 2009. 8(3): 339-350.
       
\bibitem{rs} 
        A. Rejali and M. Soleimani Malekan, \emph{Solubility of groups can be characterized by configuration}, https://arxiv.org/abs/1510.07209.
        
        \bibitem{rs1}        
        A. Rejali and M. Soleimani Malekan, \emph{Two-sided configuration of finitely presented Hopfian groups},https://arxiv.org/abs/1512.03021.
        
        
        \bibitem{ry}
        A. Rejali and A. Yousofzadeh, \emph{Group Properties Characterized by Two-sided Configurations}, Algebra Colloq. 17:4 (2010) 583-594.
        
        \bibitem{try} 
      A. Tavakoli, A. Rejali, A. Yousofzadeh, A. Abdollahi, \emph{A Note About Configuration of A Group}, Matematika, 2014, Volume 30, Number 2, 117-121.
      
      \bibitem{ytr}
A. Yousofzadeh, A. Tavakoli, A. Rejali. \emph{On configuration graph and paradoxical decomposition}. Journal of Algebra and Its Applications., 13(2): 87-98, 2014.

       \bibitem{Olshansky}
   	A. Yu. O{\'l}shanskii, \emph{Geometry of Defining Relations in Groups}, Kluwer, 1991.
   	
   	 \bibitem{sims}
  C. C. Sims, \emph{Computation with Finitely Presented Groups}. Cambridge: Cambridge University Press, 1994. 

\bibitem{walker}
    E. A. Walker, \emph{Cancellation in Direct Sums of Groups}, Proc. Amer. Math. Soc., Vol. 7, No. 5 (Oct., 1956), pp. 898-902.

\bibitem{robinson}
D. J. S. Robinson, \emph{A Course in the Theory of Groups}, 2nd edition, Springer, New York, 1983.

      
       \bibitem{osin}
       D. Osin, \emph{Small cancellations over relatively hyperbolic groups and embedding theorems}. Ann. of Math. (2) 172 (2010), no. 1, 1–39.
       
        \bibitem{segal}
   D. Segal, \emph{Polycyclic groups}, Cambridge University Press, London, 1983.

       
        \bibitem{rw}
        J. M. Rosenblatt and G. A. Willis, \emph{Weak Convergence Is Not Strong Convergence for
       Amenable Groups}, Canad. Math. Bull. 44 (2001), 231-241. MR 2002d:43001.
       
       \bibitem{fc}
   M. J. Tomkinsin, \emph{FC-groups}, 2nd edition, Pitman publishing limited, London, 1984.
       
        \bibitem{M.sapir}
  M. Sapir, \emph{Combinatorial algebra: syntax and semantics},\\
   http://www.math.vanderbilt.edu/∼msapir/book/b2.pdf.
       
       \bibitem{intro}
  O. Bogopolski, \emph{Introduction to Group Theory}, European Mathematical Society, 2008.
  
  
  \bibitem{combin}
   W. Magnus, A. Karrass, D. Solitar, \emph{Combinatorial Group Theory: Presentations of Groups in Terms of Generators and Relations}, Dover Publications, INC, New York, 2004.
       
       

\end{thebibliography}


\end{document}